\documentclass[preprint, 10.5pt, authoryear, square]{elsarticle}
\usepackage[margin=1.25in]{geometry}
\usepackage{amsmath, amsthm, amsfonts}
\usepackage{url}
\usepackage{graphicx}
\usepackage{amssymb,amsfonts}
\usepackage[all,arc]{xy}
\usepackage{enumerate}
\usepackage{mathrsfs}
\usepackage{graphicx}
\usepackage{hyperref}
\usepackage{etoolbox}
\usepackage{dsfont}
\patchcmd{\thebibliography}{\section*{\refname}}{}{}{}

\DeclareMathOperator{\N}{\mathbb{N}}
\DeclareMathOperator{\F}{\mathbb{F}}
\DeclareMathOperator{\Prob}{\mathbb{P}}

\DeclareMathOperator{\M}{\mathcal{M}}
\DeclareMathOperator{\limn}{\lim_{n \to \infty}}

\title{\textbf{Inclusion of Forbidden Minors in Random Representable Matroids}}
\ead[url]{www.princeton.edu}

\author[Computer Science Department]{Jason Altschuler \corref{mycorrespondingauthor} \fnref{fn1}}
\cortext[mycorrespondingauthor]{Corresponding Author}
\ead{jasonma@princeton.edu}\fntext[fn1]{Present address: Laboratory for Information and Decision Systems (LIDS), Massachusetts Institute of Technology, Cambridge, MA, 02139, USA.}

\author[Math Department]{Elizabeth Yang\fnref{fn2}}
\ead{eyang@princeton.edu}
\fntext[fn2]{Present address: Two Sigma Investments, New York, NY, 10013, USA.}
\date{}

\address[Computer Science Department]{Princeton University Computer Science Department, jasonma@princeton.edu}
\address[Math Department]{Princeton University Math Department, eyang@princeton.edu}

\theoremstyle{plain}
\theoremstyle{plain}
\newtheorem{thm}{Theorem}
\theoremstyle{plain}
\theoremstyle{plain}
\theoremstyle{plain}
\theoremstyle{plain}
\theoremstyle{plain}
\newtheorem{lemma}{Lemma}
\theoremstyle{plain}
\newtheorem{corol}{Corollary}
\theoremstyle{plain}
\theoremstyle{remark}
\theoremstyle{discussion}
\theoremstyle{plain}

\begin{document}



\begin{frontmatter}


\begin{abstract}
\par In 1984, Kelly and Oxley introduced the model of a random representable matroid $M[A_n]$ corresponding to a random matrix $A_n \in \F_q^{m(n) \times n}$, whose entries are drawn independently and uniformly from $\F_q$. Whereas properties such as rank, connectivity, and circuit size have been well-studied, forbidden minors have not yet been analyzed. Here, we investigate the asymptotic probability as $n \to \infty$ that a fixed $\F_q$-representable matroid $M$ is a minor of $M[A_n]$. (We always assume $m(n) \geq \text{rank}(M)$ for all sufficiently large $n$, otherwise $M$ can never be a minor of the corresponding $M[A_n]$.) When $M$ is free, we show that $M$ is asymptotically almost surely (a.a.s.) a minor of $M[A_n]$. When $M$ is not free, we show a phase transition: $M$ is a.a.s. a minor if $n - m(n) \to \infty$, but is a.a.s. not if $m(n) - n \to \infty$. In the more general settings of $m \leq n$ and $m > n$, we give lower and upper bounds, respectively, on both the asymptotic and non-asymptotic probability that $M$ is a minor of $M[A_n]$. The tools we develop to analyze matroid operations and minors of random matroids may be of independent interest.
\par Our results directly imply that $M[A_n]$ is a.a.s. not contained in any proper, minor-closed class $\mathcal{M}$ of $\F_q$-representable matroids, provided: (i) $n - m(n) \to \infty$, and (ii) $m(n)$ is at least the minimum rank of any $\F_q$-representable forbidden minor of $\mathcal{M}$, for all sufficiently large $n$. As an application, this shows that graphic matroids are a vanishing subset of linear matroids, in a sense made precise in the paper. Our results provide an approach for applying the rich theory around matroid minors to the less-studied field of random matroids.
\end{abstract}

\end{frontmatter}


\section{Introduction}

\par The motivation of this paper is to connect the study of random matroids with the rich theory recently developed around matroid minors. We ask a natural question: when does a fixed minor occur in the column dependence matroid obtained from a random matrix?
\par Formally, we consider Kelly and Oxley's model of a random representable matroid $M[A_n]$ corresponding to a random matrix $A_n \in \F_q^{m(n) \times n}$, whose entries are drawn independently and uniformly from the Galois field $\F_q$ of order $q$ \citep{KO84}. We denote this uniform distribution over random matrices in $\F_q^{m(n) \times n}$ succinctly by $[U_q]^{m(n) \times n}$, and write $A_n \sim [U_q]^{m(n) \times n}$ to indicate that $A_n$ is drawn according to it. This paper investigates the asymptotic probability as $n \to \infty$ that a fixed $\F_q$-representable matroid $M$ is a minor of $M[A_n]$.
\par Interestingly, we are able to characterize the asymptotic probability that $M$ is a minor of $M[A_n]$ solely by how fast the number of rows $m(n)$ of $A_n$ grows. Observe that $M[A_n]$ can \textit{never} have $M$ as a minor if $m(n)$ is less than the rank $r(M)$ of $M$. Thus, throughout the paper, we impose the constraint that $m(n) \geq r(M)$ for all sufficiently large $n$. 
\par We first show that every fixed \textit{free} matroid $M$ is asymptotically almost surely (a.a.s.) a minor of $M[A_n]$. We also give a closed-form expression for the non-asymptotic probability that this occurs, in terms of Gaussian coefficients.
\par However, inclusion of \textit{non-free} minors is not as simple. Formally, for any finite field $\F_q$ and any non-free, $\F_q$-representable matroid $M$, we show that the following phase transition occurs:
$$
\limn \Prob_{A_n \sim [U_q]^{m(n) \times n} } \left\{M \text{ is a minor of } M[A_n] \right\} = 
\begin{cases}
	   1 & \text{ if } n - m(n) \to \infty \\
	   0 & \text{ if } m(n) - n \to \infty
\end{cases}
$$
Along the way, we show that $M[A_n]$ is a.a.s. the free matroid on $n$ elements when $m(n) - n \to \infty$, extending a result of~\citep{KO84}.
\par We also analyze the threshold between $n - m(n) \to \infty$ and $m(n) - n \to \infty$. As will be discussed formally later -- but can be seen intuitively from the above phase transition -- whether $m(n)$ is smaller or larger than $n$ results in very different behaviors. So we investigate two cases: either $m(n) \geq n$ for all sufficiently large $n$, or $m(n) < n$ for all sufficiently large $n$.
\par In the case that $m(n) \geq n$ for all sufficiently large $n$, we show that for any non-free, $\F_q$-representable matroid $M$:
$$
\limsup_{n \to \infty} \Prob_{A_n \sim [U_q]^{m(n) \times n} } \left\{M \text{ is a minor of } M[A_n] \right\} \leq 1 - C_q
$$
where $C_q = \prod_{k=1}^{\infty} \left(1 - q^{-k}\right)$, and the limit superior is used only because the limit might not exist. (In the main text, we give intuition for the constant $C_q > 0$ by equating it to the asymptotic probability that a square matrix in $\F_q^{n \times n}$ is invertible~\citep{Cooper}.) In order to prove this asymptotic bound, we show the following non-asymptotic bound that holds for any $m(n) \geq n$:
$$
\Prob_{A \sim [U_q]^{m \times n} } \left\{M \text{ is a minor of } M[A] \right\} \leq 1 - \prod_{i=0}^{n-1} (1 - q^{i -m(n)})
$$
\par Next, in the case that $m(n) < n$ for all sufficiently large $n$, we show that provided\footnote{There is an analogue that only requires $m(n) \geq r(M)$ for sufficiently large $n$, but the resulting bound is messier.} $m(n) \geq |E|$ for all sufficiently large $n$, then for any non-free, $\F_q$-representable matroid $M = (E, I)$ with $\ell$ loops:
$$
\liminf_{n \to \infty} \Prob_{A_n \sim [U_q]^{m(n) \times n} } \left\{M \text{ is a minor of } M[A_n] \right\}
>
\left(1 - q^{-|E|}\right) p_{|E|-1, q, M}
$$
where $p_{s, q, M} \in (0, 1)$ is defined as:
$$p_{s,q,M}
=
\binom{|E|}{\ell} \left(\frac{\left(q-1\right)^{|E| - r(M) - \ell}}
{q^{s(|E| - r(M))}}\right)
\prod_{i=0}^{r(M)-1}\left(1 - q^{i-s}\right)
$$
Again, the limit inferior is used only because the limit might not exist. In order to prove this asymptotic bound, we show the following non-asymptotic bound that holds for any $m \geq r(M)$ and $n \geq |E|$:
$$
\Prob_{A \sim [U_q]^{m \times n}} \left\{M \text{ is a minor of } M[A] \right\}
>
\max_{k \in \mathbb{Z}_+, \; k \leq \min(n - |E|, m- r(M))}
\left( 1 - q^{-(n-k)}\right)
\left(1 - \left(1 - p_{m-k,q,M} \right)^{\lfloor \frac{n-k}{|E|} \rfloor}\right)
$$
\par We note that this second setting $m(n) < n$ is significantly more involved because then $A_n \in \F_q^{m(n) \times n}$ is guaranteed to have dependence relations between the columns. Intuitively, this means that we will likely require contractions (in addition to just deletions) to obtain $M$ as a minor of $M[A_n]$. But it is not even immediately clear how we should take contractions on a random matrix. The machinery we develop in order to achieve this may be of independent interest (see Section~\ref{subsec:n-bigger-outlines} for an overview of these tools).
\par Our final result allows us to leverage the connection between matroid characterizations and forbidden minors. Specifically, we show how our above results imply that, as $n \to \infty$, the random matroid $M[A_n]$ is a.a.s. not in \textit{any} fixed proper, minor-closed class $\mathcal{M}$ of $\F_q$-representable matroids, provided: (i) $n - m(n) \to \infty$, and (ii) $m(n)$ is at least the minimum rank of the $\F_q$-representable forbidden minors of $\mathcal{M}$, for all sufficiently large $n$. As an example application, this directly shows that graphic matroids are a vanishing subset of linear matroids, with respect to the uniform random distribution $[U_q]^{m(n) \times n}$, and under mild constraints on the number of rows $m(n)$.
\par Along the way, we establish various results about random representable matroids and the uniform distribution $[U_q]^{m \times n}$. We note that our techniques rely heavily on properties of the uniform distribution $[U_q]^{m \times n}$, so generalizing to other distributions over $\F_q^{m \times n}$ would likely require new machinery.

\subsection{Related work}
\par Random matrices and especially random graphs have become increasingly well understood \citep{Bollobas, JLA11, Verdu, BlakeStudholme, Tao12}, but the field of random matroid theory is still much less explored. The works of \citep{KO82a, KO82b, Oxley84, K88, KW96, KL91, KL99} analyze $\F_q$-representable random submatroids of projective geometries, while \citep{Knuth, MNWW11, LOSW13} study distributions over all matroids with fixed ground set size. Here, we consider the random representable matroid model of \citep{KO84}, which considers distributions over all $\F_q$-representable matroids with fixed ground set size. But whereas \citep{KO84} investigated properties such as rank, connectivity, and circuit size of these random representable matroids, we focus in this paper on matroid minors and inclusion in proper, minor-closed classes of $\F_q$-representable matroids. To the authors' knowledge, these are the first results to analyze matroid-minor properties of random matroids.

\subsection{Preliminaries and notation}
We refer the reader to the excellent texts \citep{Oxley} and \citep{Grimmet-Stirzaker} for background on matroid theory and probability theory, respectively. Our notation is mostly standard and adheres to theirs, but for completeness we list a few notations we use commonly throughout the paper. We note that we write ``\textit{independent}'', ``\textit{linearly independent}'', and ``\textit{stochastically independent}'' in order to disambiguate the different notions of independence in matroid theory, linear algebra, and probability theory, respectively. 
\par A matrix $A$ is a \textit{representation} of a matroid $M$ if the linear independence relations between the columns of $A$ are identical to the independence relations between the corresponding elements of $M$. A matroid $M$ is \textit{representable} over the $q$-element Galois field $\F_q$, or $\F_q$\textit{-representable} for short, if $M$ has some matrix representation $A$ over $\F_q$. We follow the notation of \citep{Oxley} to denote the matroid corresponding to a matrix $A$ by $M[A]$. We denote by $r(M)$ the \textit{rank} of the matroid $M$. A \textit{free matroid} is a matroid with all sets independent. A \textit{loop} of a matroid is an element that does not belong to any basis; for linear matroids, this is equivalent to a column being the zero vector.
\par We say that an indexed family $\{E_n\}_{n \in \N}$ of events occurs \textit{asymptotically almost surely}, or \textit{a.a.s.} for shorthand, if $\limn \Prob\{E_n\} = 1$. We write $X \sim D$ to denote that a random variable $X$ is distributed according to a distribution $D$. 

\subsection{Outline of the article}
Sections~\ref{sec:free},~\ref{sec:m-bigger}, and~\ref{sec:n-bigger} analyze the asymptotic and non-asymptotic probabilities that a fixed, $\F_q$-representable matroid $M$ is a minor of the random representable matroid $M[A_n]$ corresponding to $A_n \sim [U_q]^{m(n) \times n}$, i.e.:
\begin{align}
\Prob_{A_n \sim [U_q]^{m(n) \times n}}\{M \text{ is a minor of } M[A_n] \}\label{eq:outline-prob}
\end{align}
\par Section~\ref{sec:free} analyzes the probability in~\eqref{eq:outline-prob} when $M$ is \textit{free}. In Section~\ref{subsec:free-nonasymp}, we give a closed-form expression for the non-asymptotic setting. Section~\ref{subsec:free-asymp} analyzes the asymptotic setting as $n \to \infty$, and shows that under a mild constraint on the number of rows $m(n)$ of $A_n$, every fixed free matroid $M$ is a.a.s. a minor of $M[A_n]$.
\par It is more difficult to analyze the probability in~\eqref{eq:outline-prob} when $M$ is \textit{not free}. The settings $m(n) \geq n$ and $m(n) < n$ exhibit different phenomena and require separate tools to analyze, so we split these cases into Sections~\ref{sec:m-bigger} and~\ref{sec:n-bigger}, respectively.
\par In Section~\ref{subsec:m-bigger-infty}, we show that the probability in~\eqref{eq:outline-prob} tends a.a.s. to $0$ when $m(n) - n \to \infty$, i.e. every fixed non-free matroid $M$ is a.a.s. not a minor of $M[A_n]$. In Section~\ref{subsec:m-bigger-gen}, we analyze the general case when $m(n) \geq n$, and give asymptotic and non-asymptotic upper bounds on this probability.
\par In Section~\ref{subsec:n-bigger-nonasymp}, we present a non-asymptotic lower bound on~\eqref{eq:outline-prob} for the setting $m(n) < n$. Before proving this result, we first present two asymptotic consequences of it in Section~\ref{subsec:n-bigger-infty}: first, we give an asymptotic lower bound on~\eqref{eq:outline-prob}; and second, we show that if additionally $n - m(n) \to \infty$, then~\eqref{eq:outline-prob} tends a.a.s. to $1$, i.e. every fixed non-free matroid $M$ is a.a.s. a minor of $M[A_n]$. We then return to prove the result stated in Section~\ref{subsec:n-bigger-nonasymp}: Section~\ref{subsec:n-bigger-outlines} outlines the proof and overviews the tools we need to develop for it. We develop these tools in the following Sections~\ref{subsec:n-bigger-reducing} and~\ref{subsec:n-bigger-minor}. Informally stated, in Section~\ref{subsec:n-bigger-reducing}, we show how to apply matroid operations to ``reduce'' the size of a random representable matroid while preserving its randomness; and in Section~\ref{subsec:n-bigger-minor}, we show how to bound below the probability that a representation of a matroid is a submatrix of a random representable matroid. Finally, Section~\ref{subsec:n-bigger-proofs} combines these tools to prove the result stated in Section~\ref{subsec:n-bigger-nonasymp}.
\par Section~\ref{sec:implications} contains implications of these results regarding forbidden minors and matroid characterizations.

\section{Probability of Containing a Free Minor}\label{sec:free}
\subsection{Closed-form expression for the non-asymptotic setting}\label{subsec:free-nonasymp}
The main result of this section is a closed-form expression for the probability that a fixed free matroid $M$ is a minor of the random representable matroid $M[A]$ corresponding to $A \sim [U_q]^{m \times n}$. Clearly we may restrict to the setting that $m, n \geq r(M)$, otherwise $M$ can never be a minor of $M[A]$.
\par Our closed-form expression will be in terms of the celebrated Gaussian coefficients, which are defined as follows for $k \leq n$:
$$
{n \brack k}_q = \prod_{i=0}^{k-1} \frac{q^{n-i} - 1}{q^{k-i} - 1}
$$
for any $q$ that is a prime power~\citep{Lint}.
\begin{thm}\label{thm:free-nonasymp}
Let $\F_q$ be any finite field, and $M$ be any free matroid. If $m,n \geq r(M)$, then:
$$
\Prob_{A \sim [U_q]^{m \times n}}\{M \text{ is a minor of } M[A] \} =
q^{-mn}
\sum_{k=r(M)}^{\min(m, n)}
\left(
{\min(m, n) \brack k}_q \sum_{i=0}^k (-1)^{k - i} {k \brack i}_q q^{\max(m, n)i + \binom{k - i}{2}}
\right)
$$
\end{thm}
The proof follows immediately from the following well-known result, which counts the number of $m \times n$ matrices over a finite field $\F_q$ with rank $k$~\citep{Lint}.
\begin{lemma}\label{lem:gaussian-coeff}
If $m, n \geq k$, then the number of $m \times n$ matrices over $\F_q$ that have rank $k$ is:
$$
{\min(m, n) \brack k}_q \sum_{i=0}^k (-1)^{k - i} {k \brack i}_q q^{\max(m, n)i + \binom{k - i}{2}}
$$
\end{lemma}

A standard proof of Lemma~\ref{lem:gaussian-coeff} uses M\"{o}bius inversions on the lattice of subspaces of the vector space $\F_q^n$, and can be found in~\citep{Lint}. We now show how Lemma~\ref{lem:gaussian-coeff} directly implies Theorem~\ref{thm:free-nonasymp}.
\begin{proof}[Proof of Theorem~\ref{thm:free-nonasymp}]
$M$ is a minor of $M[A]$ if and only if $A$ contains $r(M)$ linearly independent columns, which occurs if and only if $A$ has rank at least $r(M)$. By Lemma~\ref{lem:gaussian-coeff}, there are
$$
\sum_{k=r(M)}^{\min(m,n)}
\left(
{\min(m,n) \brack k}_q \sum_{i=0}^k (-1)^{k - i} {k \brack i}_q q^{\max(m,n)i + \binom{k - i}{2}}
\right)
$$
such $m \times n$ matrices over $\mathbb{F}_q$. Since $A \sim [U_q]^{m \times n}$ is drawn from the uniform distribution, the probability that $A$ is equal to a fixed such matrix is $q^{-mn}$.
\end{proof}

\subsection{$M$ is a minor of $M[A_n]$ a.a.s. when $n \to \infty$}\label{subsec:free-asymp}
The main result of this section shows that every fixed free matroid $M$ is a.a.s. a minor of the random representable matroid $M[A_n]$ corresponding to $A_n \sim [U_q]^{m(n) \times n}$, provided only $m(n) \geq r(M)$ for all sufficiently large $n$. This is formally stated as follows.
\begin{thm}\label{thm:free}
Let $\F_q$ be any finite field, and $M$ be any free matroid. If $m: \mathbb{N} \rightarrow \mathbb{N}$ satisfies $m(n) \geq r(M)$ for all sufficiently large $n$, then:
$$
\limn \Prob_{A_n \sim [U_q]^{m(n) \times n} } \left\{M \text{ is a minor of } M[A_n] \right\} = 
1
$$
\end{thm}

Although Theorem~\ref{thm:free} can be proven by taking the limit as $n \to \infty$ of the non-asymptotic probability in Theorem~\ref{thm:free-nonasymp}, we give a simpler proof that avoids long calculations and gives a taste for the upcoming results.
\par A key ingredient of the proof is the following standard calculation of the probability that a $[U_q]^{m \times n}$ random matrix has full column rank. This appears in Lemma 3.1 from~\citep{KO84}, but for completeness we restate it below in our notation.

\begin{lemma}[Lemma 3.1 from~\citep{KO84}] \label{lem:li}
Let $m, n \in \mathbb{N}$ such that $m \geq n$. Then:
$$ \mathbb{P}_{A \sim [U_q]^{m \times n}} \{A \text{ has linearly independent columns}\}
= \prod_{i = 0}^{n-1} (1 - q^{i - m})$$
\end{lemma}
One can prove Lemma~\ref{lem:li} as a special case of Lemma~\ref{lem:gaussian-coeff}, but there is a cleaner proof that can be found in~\citep{KO84}. Since their proof is quite short, and since we will use similar techniques later in the paper, we present their proof below using our notation.
\begin{proof}
Let $v_i$ denote the $i^{\text{th}}$ column of $A$ for all $i \in \{1, \dots, n\}$. A simple calculation shows:
\begin{align}
&\;\mathbb{P}_{A \sim [U_q]^{m \times n}}\left\{A \text{ has lin. indep. columns}\right\} 
\\= &\;\prod_{i=1}^{n}\mathbb{P}_{A \sim [U_q]^{m \times n}} \left\{v_{i} \notin \text{span}\left(\{v_1, \ldots, v_{i-1} \}\right) \; \big| \; \{v_1, \ldots, v_{i - 1}\} \text{ are lin. indep.}\right\}
\\ =&\; \prod_{i = 1}^{n} (1 - q^{i - 1 - m})
\end{align}
\end{proof}
We will be particularly interested in lower bounds on this probability of $A$ having full column rank. In such situations, the following bound will often prove useful:
\begin{align}
\prod_{i = 0}^{n-1} (1 - q^{i - m})
\geq \; 1 - \sum_{i = 0}^{n-1} q^{i - m}
> \;1 - q^{n - m}\label{eq:product-sum-bound}
\end{align}
Since we will make much use of this lower bound, let us state it formally. 
\begin{corol}\label{corol:li-bound}
Let $m, n \in \mathbb{N}$ such that $m \geq n$. Then:
$$ \mathbb{P}_{A \sim [U_q]^{m \times n}} \{A \text{ has linearly independent columns}\}
> 1 - q^{n - m} $$
\end{corol}

We are now ready to prove Theorem~\ref{thm:free}.
\begin{proof}[Proof of Theorem~\ref{thm:free}]
We are given that $m(n) \geq r(M)$ for all sufficently large $n$. For these $n$, define the submatrix $B_n \sim [U_q]^{r(M) \times n}$ containing the first $r(M)$ rows of $A_n$. Applying Corollary~\ref{corol:li-bound} to $B_n^T$, $B_n$ has full row rank a.a.s. as $n \to \infty$. Whenever this occurs, there must exist $r(M)$ linearly independent columns of $B_n$. Clearly the corresponding columns from $A_n$ must also be linearly independent. Thus the submatroid of $M[A_n]$ formed by the column dependence of these columns is isomorphic to $M$, so $M$ is a minor of $M[A_n]$.
\end{proof}

\section{Probability of Containing a Non-Free Minor when $m(n) \geq n$}\label{sec:m-bigger}
We now investigate the probability that a \textit{non-free}, $\F_q$-representable matroid $M$ is a minor of $M[A_n]$, in the case that $A_n \sim [U_q]^{m(n) \times n}$ has \textit{at least as many rows as columns}.
\subsection{$M$ is a.a.s. not a minor of $M[A_n]$ when $m(n) - n \to \infty$}\label{subsec:m-bigger-infty}
The main result of this section establishes that if $m(n) - n \to \infty$, then $M$ is a.a.s. not a minor of $M[A_n]$. This is formally stated as follows.
\begin{thm}\label{thm:m-bigger-asymp}
Let $\F_q$ be any finite field, and $M$ be any non-free, $\F_q$-representable matroid. If $m: \mathbb{N} \rightarrow \mathbb{N}$ satisfies $m(n) - n \to \infty$, then:
$$
\limn \Prob_{A_n \sim [U_q]^{m(n) \times n} } \left\{M \text{ is a minor of } M[A_n] \right\} = 
0
$$
\end{thm}
In order to prove Theorem~\ref{thm:m-bigger-asymp}, we will first give a strong characterization of the asymptotic structure of the matroid $M[A_n]$ when $m(n) - n \to \infty$: it is a.a.s. the free matroid. We note that this extends Theorem 3.2 of~\citep{KO84}.
\begin{lemma}\label{lem:free-aas}
Let $\F_q$ be any finite field, $m : \N \to \N$ satisfy $m(n) - n \to \infty$, and $A_n \sim [U_q]^{m(n) \times n}$. Then $M[A_n]$ is a.a.s. the free matroid on $n$ elements.
\end{lemma}
\begin{proof}
By Corollary~\ref{corol:li-bound}, $A_n \sim [U_q]^{m(n) \times n}$ has full column rank a.a.s. as $n \to \infty$. Whenever this occurs, $M[A_n]$ is the free matroid over $n$ elements.
\end{proof}

The proof of Theorem~\ref{thm:m-bigger-asymp} now follows directly from Lemma~\ref{lem:free-aas}.

\begin{proof}[Proof of Theorem~\ref{thm:m-bigger-asymp}]
By Lemma~\ref{lem:free-aas}, $M[A_n]$ a.a.s. is the free matroid over $n$ elements, and thus cannot contain any non-free matroid $M$ as a minor.
\end{proof}

\subsection{Probability that $M$ is a minor of $M[A_n]$ in the general case when $m(n) \geq n$}\label{subsec:m-bigger-gen}
Theorem~\ref{thm:m-bigger-asymp} above gives a clean characterization of $\limn \Prob_{A_n \sim [U_q]^{m(n) \times n} } \left\{M \text{ is a minor of } M[A_n] \right\}$ when $m(n) - n \to \infty$. Here, we give asymptotic and non-asymptotic bounds for the more general case of $m(n) \geq n$.
\par We first present the non-asymptotic bound, since the asymptotic result follows directly from it.
\begin{thm}\label{thm:m-bigger-gen-nonasymp}
Let $\F_q$ be any finite field, and $M$ be any non-free, $\F_q$-representable matroid. If $m \geq n$, then:
\begin{align}
\Prob_{A \sim [U_q]^{m \times n} } \left\{M \text{ is a minor of } M[A] \right\} \leq 1 - \prod_{i=0}^{n-1} (1 - q^{i -m})
\end{align}
\end{thm}
\begin{proof}
By Lemma~\ref{lem:li}, $A$ has full column rank with probability $\prod_{i=0}^{n - 1}(1 - q^{i - m})$. Whenever this occurs, $M$ is the free matroid over $n$ elements, and thus cannot contain the non-free matroid $M$ as a minor.
\end{proof}

The main asymptotic result of this section now follows directly from taking the limit (superior) of the bound in Theorem~\ref{thm:m-bigger-gen-nonasymp}.

\begin{thm} \label{thm:m-bigger-gen-asymp}
Let $\F_q$ be any finite field, and $M$ be any non-free, $\F_q$-representable matroid. If $m: \mathbb{N} \rightarrow \mathbb{N}$ satisfies $m(n) \geq n$ for all sufficiently large $n$, then:
\begin{align}
\limsup_{n \to \infty} \Prob_{A_n \sim [U_q]^{m(n) \times n} } \left\{M \text{ is a minor of } M[A_n] \right\} \leq 1 - C_q
\end{align}
where $C_q = \prod_{k=1}^{\infty} \left(1 - q^{-k}\right)$.
\end{thm}
\begin{proof}[Proof of Theorem~\ref{thm:m-bigger-gen-asymp}] For all sufficiently large $n$, we have $m(n) \geq n$. Thus we may invoke Theorem~\ref{thm:m-bigger-gen-nonasymp} for each such $n$ to obtain the lower bound
 $\prod_{i=0}^{n-1} (1 - q^{i -m(n)}) \geq \prod_{i=0}^{n-1} (1 - q^{i-n})$, which tends to $C_q$ from above, as $n \to \infty$.
\end{proof}

Let us provide some intuition about the constant $C_q = \prod_{k=1}^{\infty} \left(1 - q^{-k}\right)$. By Lemma~\ref{lem:li}, this is precisely the limiting probability that a uniformly drawn square matrix over $\F_q$ is invertible~\citep{Cooper}:
$$
\limn \Prob_{A_n \sim [U_q]^{n \times n}}\{A \text{ nonsingular} \}
= \limn \prod_{k=1}^{n} \left(1 - q^{-k}\right) = \prod_{k=1}^{\infty} \left(1 - q^{-k}\right) = C_q
$$
Recall Euler's famous Pentagonal Number Theorem \citep{Pentagonal-Number-Theorem}, which states the following identity holds and converges absolutely for all $|x| < 1$: 
$$
\prod_{k=1}^{\infty} (1 - x^{k}) = \sum_{k=0}^{\infty} \left(-1\right)^k \left(1 - x^{2k+1}\right)x^{k(3k+1)/2}
$$
Thus in particular we have that $C_q > 1 - \frac{1}{q} - \frac{1}{q^2} > 0$ is a well-defined constant bounded away from $0$, since the size $q$ of any Galois field $\F_q$ is at least $2$.

\section{Probability of Containing a Non-Free Minor when $n > m(n)$}\label{sec:n-bigger}
We now investigate the probability that a \textit{non-free}, $\F_q$-representable matroid $M$ is a minor of $M[A_n]$, in the case that $A_n \sim [U_q]^{m(n) \times n}$ has \textit{more columns than rows}.

\subsection{A non-asymptotic lower bound on the probability $M$ is a minor of $M[A_n]$}\label{subsec:n-bigger-nonasymp}
The main result in this section is the following non-asymptotic lower bound. 

\begin{thm}\label{thm:subsec-n-bigger-gen-new}
Let $\F_q$ be any finite field, and $M = (E,I)$ be any non-free, $\F_q$-representable matroid with $\ell$ loops. If $m, n \in \mathbb{N}$ satisfy (i) $m \geq r(M)$; and (ii) $n \geq |E|$, then:
$$
\Prob_{A \sim [U_q]^{m \times n}} \left\{M \text{ is a minor of } M[A] \right\}
>
\max_{k \in \mathbb{Z}_+, \; k \leq \min(n-|E|, m-r(M))}
\left( 1 - q^{-(n-k)}\right)
\left(1 - \left(1 - p_{m-k,q,M} \right)^{\lfloor \frac{n-k}{|E|} \rfloor}\right)
$$
where $p_{s,q,M} \in (0, 1)$ is defined as:
$$p_{s,q,M}
=
\binom{|E|}{\ell} \left(\frac{\left(q-1\right)^{|E| - r(M) - \ell}}
{q^{s(|E| - r(M))}}\right)
\prod_{i=0}^{r(M)-1}\left(1 - q^{i-s}\right)
$$
\end{thm}

The proof of Theorem~\ref{thm:subsec-n-bigger-gen-new} is delayed until Section~\ref{subsec:n-bigger-proofs}, since we will first need to develop tools to analyze minors of random matroids (see Sections~\ref{subsec:n-bigger-outlines},~\ref{subsec:n-bigger-reducing}, and~\ref{subsec:n-bigger-minor}). But before describing those tools, let us first state in Section~\ref{subsec:n-bigger-infty} two direct consequences of this result for the asymptotic setting of $n \to \infty$.

\subsection{Asymptotic probability that $M$ is a minor of $M[A_n]$ when $n > m(n)$}\label{subsec:n-bigger-infty}
This section contains two results on the asymptotic probability that $M$ is a minor of $M[A_n]$ when $n > m(n)$. The first result gives a general lower bound on this probability. The second result shows that if additionally $n - m(n) \to \infty$, then this probability tends to $1$; that is, $M$ a.a.s. is a minor of $M[A_n]$. Both of these results are proved as direct corollaries of Theorem~\ref{thm:subsec-n-bigger-gen-new}, the first by setting $k(n) := m(n) + 1 - |E|$; and the second by setting $k(n) := m(n) - r(M)$.
\par Let us first present the general lower bound.

\begin{thm} \label{thm:n-bigger-gen-asymp}
Let $\F_q$ be any finite field, and $M = (E, I)$ be any non-free, $\F_q$-representable matroid with $\ell$ loops. If $m: \mathbb{N} \rightarrow \mathbb{N}$ satisfies $n > m(n) \geq |E|$ for all sufficiently large $n$, then:
$$
\liminf_{n \to \infty} \Prob_{A_n \sim [U_q]^{m(n) \times n} } \left\{M \text{ is minor of } M[A_n] \right\}
>
\left(1 - q^{-|E|}\right) p_{|E| - 1, q, M}
$$
\end{thm}
\begin{proof}
By assumption, all sufficiently large $n \in \mathbb{N}$ satisfy $m(n) \geq r(M)$ and $n \geq |E|$. For all such $n$, apply Theorem~\ref{thm:subsec-n-bigger-gen-new} with $k(n) = m(n) + 1 - |E|$. The desired result then follows by observing that $\liminf_{n \to \infty} (n - k(n)) \geq |E|$, that $p_{m(n) - k(n), q, M} = p_{|E| - 1, q, M}$, and that $1 - (1 - p)^t \geq p$ for any $p \in (0, 1)$ and $t \geq 1$. 
\end{proof}
\par Now for the second result: we show that if additionally $n - m(n) \to \infty$, then $M$ a.a.s. is a minor of $M[A_n]$. This is formally stated as follows.
\begin{thm}\label{thm:n-bigger-asymp}
Let $\F_q$ be any finite field, and $M = (E,I)$ be any non-free, $\F_q$-representable matroid. If $m: \mathbb{N} \rightarrow \mathbb{N}$ satisfies (i) $m(n) \geq r(M)$ for all sufficiently large $n$; and (ii) $n - m(n) \to \infty$, then:
$$
\limn \Prob_{A_n \sim [U_q]^{m(n) \times n} } \left\{M \text{ is a minor of } M[A_n] \right\} = 
1
$$
\end{thm}
\begin{proof}
By assumption, all sufficiently large $n \in \mathbb{N}$ satisfy $m(n) \geq r(M)$, $n \geq |E|$, and $n - |E| \geq m(n) - r(M)$. For all such $n$, apply Theorem~\ref{thm:subsec-n-bigger-gen-new} with $k(n) := m(n) - r(M)$ to obtain the lower bound:
$$
\Prob_{A \sim [U_q]^{m \times n}} \left\{M \text{ is a minor of } M[A] \right\}
>
\left( 1 - q^{-(n-k(n))}\right)
\left(1 - \left(1 - p_{r(M),q,M} \right)^{\lfloor \frac{n-k(n)}{|E|} \rfloor}\right)
$$
Both factors in the lower bound clearly tend to $1$ as $n \to \infty$, because $n- k(n) \to \infty$ and $p_{r(M), q, M}$ is independent of $n$. Since the limit of products is equal to the product of limits (if they exist), and since a probability measure is bounded above by $1$, the limiting probability of $M$ being a minor of $M[A_n]$ is $1$.
\end{proof}

\subsection{Proof outline for Theorem~\ref{thm:subsec-n-bigger-gen-new}}\label{subsec:n-bigger-outlines}
\par There are, roughly speaking, two main steps in the proof. Informally, these are: (1) finding a sequence of matroid operations on $M[A]$ that produce an ``appropriately sized'' random representable matroid $M[B]$; and (2) bounding below the probability that $M$ is a minor of $M[B]$. We develop the tools for the first step in Section~\ref{subsec:n-bigger-reducing}, and the the tools for the second in Section~\ref{subsec:n-bigger-minor}.
\par For clarity of explaining these steps, however, let us first describe the second step, since it will motivate why we need the first one to get a good final bound. For notational convenience, let us denote by $R_{t,q}(M)$ the set of $t \times |E|$ representations of an $\F_q$-representable matroid $M = (E,I)$ over $\F_q$. To start with, let us assume for simplicity that $n = |E|$. Then we can bound below the probability that $M$ is a minor of $M[A]$, by the probability that $A \in R_{m,q}(M)$. Since $A$ is drawn from the uniform distribution $[U_q]^{m \times |E|}$, $A$ is equal to a fixed element of $R_{m,q}(M)$ with probability $q^{-m|E|}$. Thus:
\begin{align}
\Prob_{A \sim [U_q]^{m \times |E|}}\{M \text{ is a minor of } M[A] \}
\geq
q^{-m|E|}\left|R_{m,q}(M) \right|
\label{eq:outline-lb}
\end{align}
Thus it suffices to bound below the number $\left|R_{m,q}(M) \right|$ of $m \times |E|$ representations of $M$ over $\F_q$. We do precisely this (see Lemma~\ref{lem:num-reps}), which immediately gives a bound for the case when $A$ is of dimension $m \times |E|$.
\par However, the story is not quite finished. Unfortunately, the number of representations $|R_{m,q}(M)|$ increases at a rate of roughly ${m \brack r(M)}_q \approx q^{mr(M)}$ as a function of $m$, which gets exponentially overrun by the $q^{-m|E|}$ factor in equation~\eqref{eq:outline-lb}. Therefore, if we do not have any additional tools, any bounds would become weaker exponentially fast in terms of $m$.
\par This motivates the first step in the proof, in which we, informally, ``reduce'' the number of rows of $A$. Specifically, we first extract from $M[A]$ a minor $M[B]$ that is a random representable matroid with $m-k \geq r(M)$ rows instead of $m$ rows. Since $M$ is a minor of $M[A]$ whenever $M$ is a minor of $M[B]$, we can then apply the above techniques and get a bound which decays in $m-k$ rather than in $m$.
\par So let us describe how to extract the minor $M[B]$ from $M[A]$. It is not immediately obvious how to do this on the \textit{random} matroid $M[A]$ because we can only control the minor we obtain when we apply contractions to \textit{deterministic} columns. (Clearly we need contractions because just applying deletions to the columns of $M[A]$ will not help us.) The simple but key idea is that because the elements of $A$ are stochastically independent, conditioning on the drawing of some elements of $A$ does not affect the distribution of any of the other elements. Thus the strategy will roughly be to draw $k$ columns of $A$, argue that they are linearly independent with some large probability (in terms of $k$, $m$, and $n$) by Corollary~\ref{corol:li-bound}, and then contract on (and delete) them. Informally, the resulting matroid is isomorphic to the random representable matroid $M[B]$, where $B \sim [U_q]^{(m-k) \times (n-k)}$ is of smaller dimension.
\par We note that in order for $M$ to be a minor of $M[B]$, we must have $m-k \geq r(M)$ and $n-k \geq |E|$. That is, $k$ must be bounded above by $\min(n - |E|, m - r(M))$.
\par Finally, we describe what happens when we let $n$ grow.\footnote{This could be viewed as a third, separate step, but because it is simple, we combine its details with the second step in Section~\ref{subsec:n-bigger-minor}.} If we partition the matrix $B$ into $t = \lfloor n/|E| \rfloor$ blocks $B_1, \dots, B_t$ of size $|E|$ (throwing away any excess columns), we know $M$ is a minor of $B$ if $M$ is a minor of any of the $B_i$. By the above, we know how to calculate the probability $p$ of the latter event for each $i$, since each $B_i$ has exactly $|E|$ columns. Since the $B_i$ are independent, we have that at least one of the $B_i$ contains $M$ as a minor with probability at least $1 - (1 - p)^t$.

\subsection{Finding a sequence of matroid operations to obtain an appropriately sized random minor of a random representable matroid}\label{subsec:n-bigger-reducing}
The main tool we develop in this section is the following lemma, which bounds below the probability that we can extract a random minor $M[B]$ from the random representable matroid $M[A]$, where $A \sim [U_q]^{m \times n}$ and $B \sim [U_q]^{(m-k) \times (n-k)}$ for $k \leq \min(m, n)$.
\begin{lemma} \label{lem:sequence}
Let $A \sim [U_q]^{m \times n}$, and $k \in \N$ satisfy $k \leq \min(m, n)$. Then with probability greater than
$$
1 - q^{k - \max(m, n)}
$$
there exists a sequence of contractions and deletions on $M[A]$ that produce a random representable matroid $M[B]$ corresponding to $B \sim [U_q]^{(m - k) \times (n - k)}$.
\end{lemma}
We note that in this paper, we will only use the above Lemma~\ref{lem:sequence} when $n > m$, in which case the sequence of matroid operations exists with probability greater than $1 - q^{k -n}$. However, we state the lemma in the more general form where $m$ can be larger than $n$, in the hope that this tool of analyzing matroid operations on random matrices is of independent interest.

A crucial ingredient in the proof of Lemma~\ref{lem:sequence} is the following fact, which states that the distribution $[U_q]^{m \times n}$ of random matrices is invariant under a change of basis.

\begin{lemma} \label{lem:change-of-basis}
If $A \sim [U_q]^{m \times n}$ and $P \in \mathbb{F}_q^{m \times m}$ is invertible, then $PA \sim [U_q]^{m \times n}$.
\end{lemma}
\begin{proof}
Denote the uniform distribution over $\F_q$ by $U_q$. The following two simple observations will be helpful. First, $cX \sim U_q$ if $X \sim U_q$ and $c \in \F_q^{\times}$. Second, $X + Y \sim U_q$ if $X, Y \sim U_q$ are stochastically independent.
\par First, we show that each $(PA)_{ij} \sim U_q$. We have $(PA)_{ij} = \sum_k P_{i, k}A_{k, j} = \sum_{k: P_{i, k} \neq 0} P_{i, k}A_{k, j}$. Observe that $|\{k: P_{i, k} \neq 0\}| > 0$, since $P$ is invertible. By the first observation above, $P_{i,k}A_{k,j} \sim U_q$ for each $k$ in the summand. Since the terms $\{P_{i,k}A_{k,j}\}$ are functions of stochastically independent random variables, they are themselves stochastically independent, and thus $\sum_{k: P_{i, k} \neq 0} P_{i, k}A_{k, j} \sim U_q$ by the second observation and a simple induction argument.
\par It remains to show that $\{(PA)_{ij}\}$ are stochastically independent. Simply observe that for all $Y \in \F_q^{m \times n}$:
$$
\Prob\{(PA)_{ij} = Y_{ij}, \; \forall i,j\}
= 
\Prob\{PA = Y \}
=
\Prob\{A = P^{-1}Y\}
=
q^{-nm}
=
\prod_{i,j} \Prob\{(PA)_{ij} = Y_{ij} \}
$$
\end{proof}

\par We are now ready to prove Lemma~\ref{lem:sequence}.

\begin{proof}[Proof of Lemma~\ref{lem:sequence}]
\par \textit{Case 1: $m > n$}. We show that there exists such a sequence of matroid operations on $A$ with probability greater than $1 - q^{k - m}$. Partition the random matrix $A$ into blocks as $A = \begin{bmatrix}
L & R
\end{bmatrix}$, where $L$ contains the left-most $k$ columns, and $R$ contains the remaining $n-k$ columns. A simple but key observation is that: because the entries of $A$ are stochastically independent, conditioning on the drawing of some entries of $A$ does not affect the distribution of any of the other entries. So draw the entries of $L$, but leave $R$ as a random matrix. As we will see, this allows us to preserve randomness in the resulting matroid minor.
\par By applying Corollary~\ref{corol:li-bound} to $A^T$, we know that the columns of $L$ are linearly independent with probability greater than $1 - q^{k -m}$. Whenever this occurs, we may apply the following operations to $A$: apply a possible change of basis that maps these first $k$ columns to the standard basis vectors $e_1, \dots, e_k$; then contract by them (and delete them). Observe that contraction (and deletion) of a unit column corresponds to deleting that column as well as the row containing the non-zero entry. Thus by Lemma~\ref{lem:change-of-basis}, the resulting random matrix is of the form $B \sim [U_q]^{(m - k) \times (n - k)}$. 
\par \textit{Case 2: $m \leq n$}. Now we show that there exists such a sequence of matroid operations on $A$ with probability greater than $1 - q^{k - n}$. This will take slightly more care because we (of course) cannot contract a column-dependence matroid by its rows.
\par Partition the random matrix $A$ into blocks as follows:
\begin{align}
A = \begin{bmatrix}
G \sim [U_q]^{k \times n} \\ 
H \sim [U_q]^{(m - k) \times n} \notag
\end{bmatrix}
\end{align}
As in Case 1 above, we will draw some entries of $A$ but not all of them, in order to preserve randomness in the resulting matroid minor. Specifically, draw the entries of $G$, but leave $H$ as a random matrix for now.
\par By applying Corollary~\ref{corol:li-bound} to $G^T$, we know that the rows of $G$ are linearly independent with probability greater than $1 - q^{k -n}$. Because row rank equals column rank for matrices, this would imply the existence of a linearly independent subset of $k$ columns in $G$. Without loss of generality, we may assume that these are the first $k$ columns of $G$, giving us the following picture:
\begin{align}
A = \begin{bmatrix}
W \in \F_q^{k \times k} & X \in \F_q ^{k \times (n - k)} \\ 
Y \sim [U_q]^{(m - k) \times k} & Z \sim [U_q]^{(m - k) \times (n - k)} \notag
\end{bmatrix}
\end{align}
where $W$ and $X$ are drawn, $Y$ and $Z$ are still random, and $W$ is full rank. Now draw the entries in $Y$. By a basic property of linear algebra, the first $k$ columns of $A$ (i.e. those corresponding to the columns of $W$) are linearly independent regardless of the value of $Y$.
\par Therefore, with probability greater than $1 - q^{k-n}$, we may apply the following operations to $A$: apply a possible change of basis that maps these first $k$ columns to the standard basis vectors $e_1, \dots, e_k$; then contract by them (and delete them). By an identical argument to the one used in case 1 above, the resulting random matrix is of the form $B \sim [U_q]^{(m - k) \times (n - k)}$.
\end{proof}

\subsection{Lower bounding the probability of containing a minor, by counting representations over $\F_q$}\label{subsec:n-bigger-minor}
In this section, we provide lower bounds on the probability that a random representable matroid contains a given minor. The main result of this section is the following.

\begin{lemma}\label{lem:lower-bound}
Let $\F_q$ be a finite field, and $M = (E, I)$ be any $\F_q$-representable matroid with $\ell$ loops. Then for any $m \geq r(M)$ and $n \geq |E|$:
$$\Prob_{A \sim [U_q]^{m \times n}} \left\{M \text{ is a minor of } M[A] \right\}
\geq
1 - \left(1 - p_{m, q, M} \right)^{\lfloor \frac{n}{|E|} \rfloor}
$$
\end{lemma}

The main tool we will use to prove Lemma~\ref{lem:lower-bound} is the following lower bound on the number of representations a matroid has over $\F_q$.

\begin{lemma}\label{lem:num-reps}
Let $\F_q$ be any finite field, $M = (E, I)$ be any $\F_q$-representable matroid with $\ell$ loops, and $m \geq r(M)$. There are at least $\binom{|E|}{\ell} \left(q-1\right)^{|E| - r(M) - \ell}\prod_{i=1}^{r(M)} \left(q^{m} - q^{i-1} \right)$ representations of $M$ over $\F_q$ of dimension $m \times |E|$.
\end{lemma}
\begin{proof}
By the assumptions that $M$ is $\F_q$-representable and $m \geq r(M)$, there exists some representation $R \in \F_q^{m \times |E|}$ of $M$. Fix any basis $S \subseteq E$ of $M$; then the corresponding set of columns $R[S]$ spans the column space of $R$. For any set $S'$ of $|S| = r(M)$ linearly independent vectors in $\F_q^{m}$, consider any invertible linear map $P_{S'} : \F_q^m \to \F_q^m$ that sends the columns $R[S]$ to the columns in $S'$. Since a change of basis clearly does not affect the linear independence of columns, each matrix $P_{S'}R$ is a valid representation of $M$. Further, the representations $P_{S'}R$ are clearly distinct for distinct sets $S'$, regardless of which mappings $P_{S'}$ were chosen. Thus since there are $\prod_{i=1}^{r(M)} \left(q^m - q^{i-1}\right)$ such sets $S'$, there are at least $\prod_{i=1}^{r(M)} \left(q^m - q^{i-1}\right)$ representations $R \in \F_q^{m \times |E|}$ of $M$, none of which send the columns $R[S]$ to the same columns $P_{S'}R[S]$.
\par Next, for each such representation $P_{S'}R$, we can multiply each non-zero column (equivalently, each column corresponding to a non-loop element of $M$) that is not in the basis $S'$ by any element of $\F_q^{\times}$ and still be a representation of $M$. Since there are precisely $|E| - r(M) - \ell$ of these columns, there are at least $(q-1)^{|E| - r(M) - \ell}\prod_{i=1}^{r(M)} \left(q^m - q^{i-1}\right)$ representations $R \in \F_q^{m \times |E|}$ of $M$. 
\par Finally, we can introduce a factor of $\binom{|E|}{\ell}$ to account for the ordering of the columns. If we treat all non-zero columns as one type of column, and all zero columns (loops) as another type, we see there are $\binom{|E|}{\ell}$ distinct ways to arrange the zero and non-zero columns. 
\end{proof}
\par We are now ready to prove Lemma~\ref{lem:lower-bound}.
\begin{proof}[Proof of Lemma~\ref{lem:lower-bound}]
Denote $t = \lfloor \frac{n}{|E|} \rfloor$. Partition the first $t|E|$ columns of $A$ into $t$ blocks of size $|E|$, and denote the resulting submatrices by $A_1, \dots, A_t \in \F_q^{m \times |E|}$. By Lemma~\ref{lem:num-reps}, there are at least $\binom{|E|}{\ell} \left(q-1\right)^{|E| - r(M) - \ell} \prod_{i=1}^{r(M)}(q^m - q^{i-1})$ representations $R \in \F_q^{m \times |E|}$ of $M$, each occuring with probability $q^{-m |E|}$ when drawn from the uniform distribution $[U_q]^{m \times |E|}$. Therefore, for each $i \in \{1, \dots, t\}$, the probability that $M$ is a minor of $M[A_i]$ is bounded below by $p_{m, q, M}$.
\par Now, since the events that $M$ is a minor of $M[A_i]$ are stochastically independent, $M$ is a minor of at least one of the $M[A_i]$ with probability at least $1 - \left(1 - p_{m, q, M})\right)^t$. This completes the proof since $M$ is a minor of $M[A]$ whenever it is a minor of one of the $M[A_i]$.
\end{proof}

\subsection{Proof of Theorem~\ref{thm:subsec-n-bigger-gen-new}}\label{subsec:n-bigger-proofs}
Now that we have tools to analyze minors of random representable matroids, we are finally ready to prove Theorem~\ref{thm:subsec-n-bigger-gen-new}. The proof formalizes the intuition given earlier about how to bound below the probability that $M$ is a minor of $M[A]$.

\begin{proof}[Proof of Theorem~\ref{thm:subsec-n-bigger-gen-new}]
It suffices to show that the desired inequality holds for each positive integer $k \leq \min(n - |E|, m- r(M))$. So fix any such $k$. Applying Lemma~\ref{lem:sequence} to $A$, we have that with probability greater than $1 -q^{-(n-k)}$, there exists a sequence $S$ of contractions and deletions on $M[A]$ that result in a linear matroid $M[B]$, where $B \sim [U_q]^{(m-k) \times (n-k)}$. Therefore, by conditioning on whether such a sequence $S$ exists:
\begin{align}
\; &\Prob_{A \sim [U_q]^{m \times n}}\{M \text{ is minor of } M[A]\}
\\ >& \left(1 -q^{-(n-k)}\right) \cdot \Prob_{A \sim [U_q]^{m \times n}} \{M \text{ is minor of } M[A] \; \Big| \; \exists \text{ sequence } S\}
\\ \geq & \left(1 -q^{-(n-k)}\right)  \cdot \Prob_{B \sim [U_q]^{(m-k) \times (n-k)}}\{M \text{ is minor of } M[B]\} \label{eq:proof-main-minor-of-minor}
\\ \geq & \left(1 - q^{-(n-k)}\right) \cdot
\left(1 - \left(1 -
p_{m-k, q, M}
\right)^{\lfloor \frac{n-k}{|E|} \rfloor}
\right)
\label{eq:proof-main-multinomial}
\end{align}
where the inequality in~\eqref{eq:proof-main-minor-of-minor} is due to the fact that a minor of a minor of a matroid is also a minor of that matroid, and the inequality in~\eqref{eq:proof-main-multinomial} follows by an application of Lemma~\ref{lem:lower-bound}.
\end{proof}

\section{Implications about Inclusion of Large Random Representable Matroids in Proper, Minor-Closed Classes of $\F_q$-Representable Matroids}\label{sec:implications}
Because the study of matroid minors is closely entwined with the study of matroid characterizations, we can obtain information about the matroid class of a $[U_q]^{m(n) \times n}$ random matrix by considering forbidden-minor characterization theorems.
\par Combining our results in Theorems~\ref{thm:free} and~\ref{thm:n-bigger-asymp} directly gives that for every finite field $\F_q$ and every fixed proper, minor-closed class $\M$ of $\F_q$-representable matroids, $\M$ is a vanishingly small subset of linear matroids with respect to the distribution $[U_q]^{m(n) \times n}$ (under mild assumptions on $m : \N \to \N$). This result is formally stated as follows.

\begin{thm}\label{thm:minor-closed}\footnote{We present Theorem~\ref{thm:minor-closed} in the setting $n - m(n) \to \infty$, which we know gives clean bounds by Theorems~\ref{thm:free} and~\ref{thm:n-bigger-asymp}. We could certainly obtain similar results for other settings of $m : \N \to \N$ using Theorem~\ref{thm:n-bigger-gen-asymp}, but the bounds and results would not be as elegant.}
Let $\F_q$ be a finite field. Consider any proper, minor-closed class $\M$ of $\F_q$-representable matroids, and let $k_{\M}$ be the minimum rank of any $\F_q$-representable excluded minor of $\M$. If $m: \mathbb{N} \rightarrow \mathbb{N}$ satisfies (i) $m(n) \geq k_{\M}$ for all sufficiently large $n$; and (ii) $n - m(n) \to \infty$, then:
$$
\limn \Prob_{A_n \sim [U_q]^{m(n) \times n}} \{ M[A_n] \not\in \M \} = 1
$$
\end{thm}
\begin{proof}
Let $M$ be an $\F_q$-representable excluded minor of $\M$ with rank $k_{\M}$. Then the probability that $M[A_n] \not\in \M$ is bounded below by the probability that $M$ is a minor of $M[A_n]$. By Theorems~\ref{thm:free} and~\ref{thm:n-bigger-asymp}, the latter probability tends to $1$ as $n \to \infty$, since by assumption $m(n) \geq k_{\M} = r(M)$ for all sufficiently large $n$.
\end{proof}


\par The power of Theorem~\ref{thm:minor-closed} is that it can easily be combined with any known forbidden-minor characterization. For example, we can show that graphic matroids are a vanishing subset of linear matroids, with respect to the uniform random distribution $[U_q]^{m(n) \times n}$, and under mild constraints on $m : \N \to \N$.
\par This is an interesting result in itself. It is known that graphic matroids are a subset of linear matroids, since every graph can be represented as a matrix (its oriented incidence matrix), but not every matrix can be represented as a graph (Tutte's Theorem gives necessary and sufficient conditions \citep{Tutte59}). However, it is not obvious how frequently a random linear matroid is graphic. This is given by the following corollary.

\begin{corol}\label{corol:graph}
Let $\F_q$ be any finite field, and let $m: \mathbb{N} \rightarrow \mathbb{N}$ satisfy $n - m(n) \to \infty$. Then
$$\lim_{n \to \infty} \mathbb{P}_{A_n \sim [U_q]^{m(n) \times n}}\{ M[A_n] \text{ is not a graphic matroid}\} = 1$$
if for all sufficiently large $n$: $m(n) \geq 2$ if $q > 2$ or $m(n) \geq 3$ if $q = 2$.
\end{corol}
\begin{proof}
By Theorem~\ref{thm:minor-closed}, it suffices to compute $k_{\M}$ the minimum rank of $\F_q$-representable excluded minors of graphic matroids. To do this, recall Tutte's characterization of graphic matroids, which states that a matroid is graphic if and only if it does not contain as a minor any of $U_{2,4}$, $F_{7}$, $F_{7}^*$, $M^*(K_5)$, and $M^*(K_{3, 3})$ \citep{Tutte59}, where $M(G)$ denotes the matroid corresponding to a graph $G$. The only matroid of these that has rank $2$ is $U_{2,4}$, which is $\F_q$-representable only for $q > 2$~\citep{Tutte2}. Thus the rank $k_{\M}$ of the smallest $\F_q$-representable excluded minor is either rank$(U_{2,4}) = 2$ if $q > 2$, or rank$(F_7) = 3$ if $q = 2$.
\end{proof}

\section*{Acknowledgements}
We are indebted to the two anonymous reviewers for their many insightful suggestions that have greatly strengthened the paper. We thank Emmanuel Abbe for his helpful feedback, as well as suggesting the problem of how often large random binary matrices can be graphic matroids (answered in Corollary~\ref{corol:graph}), during his seminar on coding theory and random graphs at Princeton University. We thank Paul Seymour and Ramon van Handel for helpful conversations about graph theory and random matrix theory, respectively.
\par JA was supported by NSF Graduate Research Fellowship 1122374.

\newpage

{\footnotesize
\bibliography{Inclusion_of_Forbidden_Matroid_Minors_in_Random_Binary_Matricesbib}{}}

\begin{thebibliography}{}

\bibitem[Apostol, 1976]{Pentagonal-Number-Theorem}
Apostol, T.~M. (1976).
\newblock {\em Introduction to Analytic Number Theory}.
\newblock Undergraduate Texts in Mathematics. Springer Science+Business Media,
  Inc.

\bibitem[Blake and Studholme, 2006]{BlakeStudholme}
Blake, I. and Studholme, C. (2006).
\newblock Properties of random matrices and applications.
\newblock {\em Available online at \url{http://www. cs. toronto.
  edu/\~cvs/coding}}.

\bibitem[Bollob{\'a}s, 1998]{Bollobas}
Bollob{\'a}s, B. (1998).
\newblock {\em Random Graphs}.
\newblock Springer.

\bibitem[Cooper, 2000]{Cooper}
Cooper, C. (2000).
\newblock On the rank of random matrices.
\newblock {\em Random Structures and Algorithms}, 16:209--232.

\bibitem[Grimmett and Stirzaker, 1992]{Grimmet-Stirzaker}
Grimmett, G.~R. and Stirzaker, D.~R. (1992).
\newblock {\em Probability and Random Processes}.
\newblock Oxford Science Publications. Clarendon Press.

\bibitem[Janson et~al., 2000]{JLA11}
Janson, S., {\L}uczak, T., and Rucinski, A. (2000).
\newblock {\em Random Graphs}, volume~45.
\newblock John Wiley \& Sons.

\bibitem[Kelly and Oxley, 1982a]{KO82a}
Kelly, D.~G. and Oxley, J.~G. (1982a).
\newblock Asymptotic properties of random subsets of projective spaces.
\newblock {\em Mathematical Proceedings of the Cambridge Philosophical
  Society}, 91:119--130.

\bibitem[Kelly and Oxley, 1982b]{KO82b}
Kelly, D.~G. and Oxley, J.~G. (1982b).
\newblock Threshold functions for some properties of random subsets of
  projective spaces.
\newblock {\em The Quarterly Journal of Mathematics, Second Series},
  33:463--469.

\bibitem[Kelly and Oxley, 1984]{KO84}
Kelly, D.~G. and Oxley, J.~G. (1984).
\newblock On random representable matroids.
\newblock {\em Studies in Applied Mathematics}, 71:181--205.

\bibitem[Knuth, 1975]{Knuth}
Knuth, D.~E. (1975).
\newblock Random matroids.
\newblock {\em Discrete Mathematics}, 12:341--358.

\bibitem[Kordecki, 1988]{K88}
Kordecki, W. (1988).
\newblock Strictly balanced submatroids in random subsets of projective
  geometries.
\newblock {\em Colloquium Mathematicae}, 55(2):371--375.

\bibitem[Kordecki, 1996]{KW96}
Kordecki, W. (1996).
\newblock Small submatroids in random matroids.
\newblock {\em Combinatorics, Probability \& Computing}, 5:257--266.

\bibitem[Kordecki and {\L}uczak, 1991]{KL91}
Kordecki, W. and {\L}uczak, T. (1991).
\newblock On random subsets of projective spaces.
\newblock {\em Colloquium Mathematicae}, 62(2):353--356.

\bibitem[Kordecki and {\L}uczak, 1999]{KL99}
Kordecki, W. and {\L}uczak, T. (1999).
\newblock On the connectivity of random subsets of projective spaces.
\newblock {\em Discrete Mathematics}, 196:207--217.

\bibitem[Lowrance et~al., 2013]{LOSW13}
Lowrance, L., Oxley, J., Semple, C., and Welsh, D. (2013).
\newblock On properties of almost all matroids.
\newblock {\em Advances in Applied Mathematics}, 50(1):115--124.

\bibitem[Mayhew et~al., 2011]{MNWW11}
Mayhew, D., Newman, M., Welsh, D. J.~A., and Whittle, G. (2011).
\newblock On the asymptotic proportion of connected matroids.
\newblock {\em European Journal of Combinatorics}, 32:882--890.

\bibitem[Oxley, 1984]{Oxley84}
Oxley, J.~G. (1984).
\newblock Threshold distribution functions for some random representable
  matroids.
\newblock {\em Mathematical Proceedings of the Cambridge Philosophical
  Society}, 95:335--347.

\bibitem[Oxley, 1992]{Oxley}
Oxley, J.~G. (1992).
\newblock {\em Matroid Theory}, volume~3.
\newblock Oxford University Press, New York.

\bibitem[Tao, 2012]{Tao12}
Tao, T. (2012).
\newblock Topics in random matrix theory.
\newblock {\em Graduate Studies in Mathematics}, 132.

\bibitem[Tulino and Verd{\'u}, 2004]{Verdu}
Tulino, A. and Verd{\'u}, S. (2004).
\newblock {\em Random Matrix Theory and Wireless Communications}, volume~1.
\newblock Now Publishers Inc.

\bibitem[Tutte, 1959]{Tutte59}
Tutte, W.~T. (1959).
\newblock Matroids and graphs.
\newblock {\em Transactions of the American Mathematical Society}, 90:527--552.

\bibitem[Tutte, 1965]{Tutte2}
Tutte, W.~T. (1965).
\newblock Lectures on matroids.
\newblock {\em Journal of Research of the National Bureau of Standards},
  69B(1-47):468.

\bibitem[van Lint and Wilson, 1992]{Lint}
van Lint, J.~H. and Wilson, R.~M. (1992).
\newblock {\em A Course in Combinatorics}.
\newblock Cambridge University Press.

\end{thebibliography}
\bibliographystyle{apalike}

\end{document}